\documentclass[11pt]{amsart}

\usepackage{latexsym}
\usepackage{amssymb}
\usepackage{amsmath}

\usepackage[italian, english]{babel}

\newtheorem{theorem}{Theorem}[section]

\newtheorem{proposition}[theorem]{Proposition}
\newtheorem{corollary}[theorem]{Corollary}

\theoremstyle{definition}

\theoremstyle{remark}
\newtheorem{remark}[theorem]{Remark}

\def\N{\mathbb{N}}

\def\F{\mathcal{F}}

\def\U{\mathcal{U}}
\def\V{\mathcal{V}}

\begin{document}

\title[Monochromatic exponential patterns]
{Monochromatic exponential triples: 
\\
an ultrafilter proof}

\author{Mauro Di Nasso}

\author{Mariaclara Ragosta}

\address{Dipartimento di Matematica\\
Universit\`a di Pisa, Italy}

\email{mauro.di.nasso@unipi.it}
\email{mariaclara.ragosta@phd.unipi.it}



\begin{abstract}
We present a short ultrafilter proof of the existence of
mono\-chromatic exponential triples $\{a, b, b^a\}$ in any finite coloring of
the natural numbers. 
The proof is given from scratch and uses only Ramsey's theorem, 
the notion of asymptotic density and the definition of ultrafilter as prerequisites.
We then generalize the construction using a special ultrafilter
whose existence is well known in the algebra of ultrafilters,
and prove a new result on the existence of infinite
monochromatic exponential patterns.
\end{abstract}

\maketitle

\maketitle

\section*{Introduction}

Arithmetic Ramsey theory is a branch of combinatorics
where one looks for monochromatic patterns
in the natural numbers (or in the integers) that are defined by 
arithmetic operations.
Probably the simplest (but still relevant) example goes back to 1918, when 
I. Schur \cite{sh} proved
that in any finite coloring of the natural numbers
there is always a monochromatic triple $a, b, a+b$.\footnote
{In fact, he proved the following finite version of that property:
For every $r$ there exists $N$
such that in every $r$-coloring $\{1,\ldots,N\}=C_1\cup\ldots\cup C_r$
one finds a monochromatic triple $\{a, b, a+b\}\subseteq C_i$.}
Another classical result in this area is 
\emph{van der Waerden's Theorem} \cite{vdW}:
In any finite coloring of the natural numbers,
one finds arbitrarily long arithmetic progressions
$a, a+d, \ldots, a+\ell d$ that are monochromatic.

In the recent article \cite{sa}, J. Sahasrabudhe proved the existence 
of several monochromatic models defined by means of
multiplications and exponentiations.
His general results require rather complex combinatorial arguments, 
so he considered the simple case of exponential triples as a 
starting point to illustrate his proof:

\begin{itemize}
\item
\emph{Theorem.} 
In any finite coloring 
$\N=C_1\cup\ldots\cup C_r$ there exists
a monochromatic exponential triple $\{a, b, b^a\}\subseteq C_i$.
\end{itemize}

However, even in this simple case, nontrivial tools 
and notions are needed, namely repeated applications 
of the finite version of \emph{van der Waerden's Theorem} and 
the \emph{compactness property} on the convergence of (sub)sequences 
in the topological space $\{1,\ldots,k\}^\N$ of $k$-colorings.

The goal of this paper is to show that, as is the case with other 
problems in Ramsey arithmetic theory, the tool of ultrafilters 
can also be useful for finding exponential patterns. 
In fact, a direct use of ultrafilters combined with a basic ingredient 
of Sahasrabudhe's argument, namely the use of the function $f(n,m)=2^nm$, 
yields a few-line proof of the above theorem 
(see Corollary \ref{corollary1b}).

In the second part of the paper, we extend the construction 
and show that the ultrafilter technique can also produce \emph{infinite}
monochromatic exponential patterns (see Corollary \ref{corollary2b}).
This is to be contrasted with the fact that the monochromatic patterns found in \cite{sa}, 
while possessing a richer combinatorial structure that also involves multiplications, 
are all finite in nature.

This article is organized into two sections.
In Section 1, the existence of monochromatic exponential triples $a, b, b^a$
is proved from scratch; in particular, no use is made of van der Waerden's Theorem 
or results from the algebra of ultrafilters.
As a first step we prove that, given any ultrafilter with a specified property $(\star)$ 
we obtain monochromatic patterns of the form $\{x, y, 2^xy\}$, 
and thus also of the form $\{a, b, b^a\}$. Then, we prove the existence of 
such ultrafilters using only simple properties 
of asymptotic density and of sets of differences.

In Section 2, we consider ultrafilters satisfying $(\dagger)$, 
a property stronger than $(\star)$. 
Their existence is not a problem, since it follows directly 
from \emph{Brauer's Theorem} \cite{br}, an improved version 
of van der Waerden's Theorem in which even the common 
difference of the monochromatic arithmetic progression belongs to the same color.\footnote
{~For those familiar with the algebra of ultrafilters, it is worth mentioning 
that every idempotent minimal ultrafilter of $(\beta\N,\oplus)$ satisfies the 
desired property $(\dagger)$.}
Extending the arguments used in Section 1 by an inductive process, 
we construct a monochromatic infinite sequence with the property 
that all appropriate exponentiations of its elements belong to that same color. 

\smallskip
Following the usual terminology of Ramsey Theory, 
in the following by ``finite coloring'' of a set $X$ is meant 
a finite partition $X=C_1\cup\ldots\cup C_r$, where the pieces 
$C_i$ are called ``colors.'' 
A subset $A$ is ``monochromatic'' with respect to the
coloring $X=C_1\cup\ldots\cup C_r$ if $A\subseteq C_i$ for some $i$.

Ramsey theory owes its name to Ramsey's Theorem \cite{ra}, 
a fundamental classical result in combinatorics. 
We recall below its formulation for pairs. 
(It is the only result of the theory that will be used in Section 1).

\smallskip
\noindent
\emph{Ramsey's Theorem} (1930).
Let $[\N]^2=C_1\cup\ldots\cup C_r$ be a finite coloring
of the set of pairs of natural numbers. Then there exists an 
infinite set $H$ which is ``homogeneous" for that coloring, that is,
its set of pairs $[H]^2\subseteq C_i$ is monochromatic.

\medskip
\section{Monochromatic exponential triples: a self-contained proof}

As noted in the introduction, this section is entirely self-contained 
in that it is not based on any relevant results from arithmetic Ramsey Theory or
from the algebra of ultrafilters, with the only exception of 
Ramsey's Theorem for pairs.
Let us first recall the notion of an ultrafilter.\footnote
{~See,\emph{e.g.}, Chapter 3 of \cite{hs}.} 

\smallskip
A \emph{filter} $\F$ on a set $I$ is a nonempty proper family of subsets of $I$ 
that is closed under supersets and under finite intersections. Precisely:
\\
(1) $\emptyset\notin\F$ and $I\in\F$;
(2) $B\supseteq A\in\F\Rightarrow B\in\F$;
(3) $A,B\in\F\Rightarrow A\cap B\in\F$.

By a direct application of \emph{Zorn's Lemma}, we obtain the existence 
of filters that are maximal with respect to inclusion; such filters are called \emph{ultrafilters}.
It is easy to show that an ultrafilter $\U$ is a filter that satisfies the following additional property:
(4) $A\notin\U\Rightarrow A^c\in\U$.

Ultrafilters are particularly relevant in Ramsey Theory because they satisfy the 
following Ramsey property:
``If $\U$ is an ultrafilter on the set $I$ and
$I=C_1\cup\ldots\cup C_r$ is a finite coloring, then exactly one of the colors $C_i\in\U$."

Two basic methods for producing new ultrafilters from given ultrafilters are as follows.

\begin{enumerate}
\item
If $\U$ is an ultrafilter on $I$ and $f:I\to J$ is a function,
the \emph{image ultrafilter} of $\U$ under $f$ is the ultrafilter
$f_*(\U)$ on $J$ defined by setting for every $B\subseteq J$:
$$B\in f_*(\U)\ \Leftrightarrow\ f^{-1}(B)=\{i\in I\mid f(i)\in B\}\in\U.$$
\item
If $\U$ and $\V$ are ultrafilters on the sets $I$ and $J$ respectively,
their \emph{tensor product} $\U\otimes\V$ 
is the ultrafilter on the Cartesian product
$I\times J$ defined by setting for every $X\subseteq I\times J$:
$$X\in\U\otimes\V\ \Leftrightarrow\ \{i\in I\mid \{j\in J\mid (i,j)\in X\}\in\V\}\in\U.$$
\end{enumerate}

That the above families $f_*(\U)$ and $\U\otimes\V$ actually satisfy 
the properties of an ultrafilter can be verified in a straightforward manner.

The following result incorporates into the ultrafilter framework 
the basic ingredient of Sahasrabudhe's proof, namely the use of the 
function $(n,m)\mapsto 2^nm$.

\begin{theorem}\label{theorem1}
Let $f:\N\times\N\to\N$ be the function $f(n,m)=2^n m$
and let $\U$ be any ultrafilter on $\N$ that satisfies the following property:
\begin{itemize}
\item[($\star$)]
If $X\in\U$ then for every $\ell$ there exists
a triple $\{b,c,b+\ell c\}\subseteq X$.
\end{itemize}
Then for every $A\in f_*(\U\otimes\U)$
there exists a triple $\{x, y, 2^x y\}\subseteq A$.
\end{theorem}

\begin{proof}
By the definitions, $A\in f_*(\U\otimes\U)\Leftrightarrow
\widehat{A}:=\{n\in\N\mid A/{2^n}\in\U\}\in\U$, where we denoted
$A/{2^n}:=\{m\in\N\mid 2^n m\in A\}$.
Pick any $a\in\widehat{A}$, so that $A/{2^a}\in\U$. By the property $(\star)$,
there exist
$b, c, b+2^a c\in A/{2^a}\cap\widehat{A}\in\U$.
Since $A/{2^b}, A/{2^{b+2^a c}}\in\U$, we can 
pick an element $d\in A/{2^b}\cap A/{2^{b+2^a c}}$.
Finally, let $x:=2^a c$ and $y:=2^b d$.
Then $x\in A$ since $c\in A/{2^a}$; $y\in A$ since $d\in A/{2^b}$;
and $2^xy=2^{2^ac}\,2^bd=2^{b+2^ac}\,d\in A$
since $d\in A/{2^{b+2^ac}}$.
\end{proof}

\begin{corollary}\label{corollary1a}
In any finite coloring $\N=C_1\cup\ldots\cup C_r$ there exists
a monochromatic pattern $\{x, y, 2^x y\}\subseteq C_i$.
\end{corollary}

\begin{proof}
Pick an ultrafilter $\U$ with property $(\star)$
and let $C_i$ be the color that belongs to 
the ultrafilter $f_*(\U\otimes\U)$.
\end{proof}

\begin{corollary}\label{corollary1b}
In any finite coloring $\N=C_1\cup\ldots\cup C_r$ there exists
a monochromatic exponential pattern $\{a, b, b^a\}\subseteq C_i$.
\end{corollary}

\begin{proof}
Let $\N=D_1\cup\ldots\cup D_r$ be the coloring
where $D_j=\{n\mid 2^n\in C_j\}$. By the previous corollary there
exists a monochromatic triple $\{x, y, 2^x y\}\subseteq D_i$.
If we let $a:=2^x$, $b:=2^y$, and $c:=2^{2^x y}$, we have
that $\{a, b, c\}\subseteq C_i$ where $c=b^a$.
\end{proof}

The existence of ultrafilters satisfying the property $(\star)$ above 
is a well-known fact.
For example, those familiar with the algebra on the space of ultrafilters 
know that every idempotent minimal ultrafilter of $(\beta\N,\oplus)$ 
satisfies combinatorial properties much stronger than $(\star)$.\footnote
{~Sets belonging to a minimal idempotent ultrafilter are called \emph{central}; 
they are much studied in Ramsey theory because of their combinatorial richness.
See, \emph{e.g.}, Chapter 14 of \cite{hs}.}

However, the existence of such ultrafilters can be proved 
from scratch by simple arguments using only Ramsey's theorem 
for pairs and the basic properties of the asymptotic density.
We show the details below in order to fulfill the promise of keeping 
this section self-contained.

We now recall the fundamental notion of asymptotic density for sets of natural numbers.

The \emph{asymptotic lower density} $\underline{d}(A)$
and the \emph{asymptotic upper density} $\overline{d}(A)$
of a set $A\subseteq\N$ are defined by letting, respectively:
$$\underline{d}(A):=\liminf_{n\to\infty}\frac{|A\cap[1,n]|}{n}\quad\text{and}\quad
\overline{d}(A):=\limsup_{n\to\infty}\frac{|A\cap[1,n]|}{n}.$$

It directly follows from the definitions that
$0\le\underline{d}(A)\le\overline{d}(A)\le 1$, and that
$\underline{d}(A^c)=1-\overline{d}(A)$.
Notice that if $\underline{d}(A)=\underline{d}(B)=1$ then $\underline{d}(A\cap B)=1$,
and if $\overline{d}(A),\overline{d}(B)=0$ then $\overline{d}(A\cup B)=0$.

The \emph{set of differences} of a set $X\subseteq\N$ is
the set 
$$\Delta(X):=\{x'-x\mid x',x\in X, x'>x\}.$$
A set $A\subseteq\N$ is called a $\Delta$-\emph{set} if 
$A\supseteq\Delta(X)$ for some infinite $X$. 
A set $A\subseteq\N$ is called \emph{thick}
if it includes arbitrarily long intervals.

The following are well-known basic properties, but,
as promised, we include (elementary) proofs.

\begin{proposition}
\

\begin{enumerate}
\item
If the set $A$ has upper density $\overline{d}(A)=1$
then $A$ is thick.
\item
Every thick set is a $\Delta$-set.
\item
If $A\subseteq\N$ have positive upper density then for every infinite set $X$,
the intersection $\Delta(A)\cap\Delta(X)$ is nonempty.
\item
The family of $\Delta$-sets has the Ramsey property:
If $A=C_1\cup\ldots\cup C_r$ is a finite coloring of
a $\Delta$-set $A$ then one of the colors $C_i$ is a $\Delta$-set.
\end{enumerate}
\end{proposition}

\begin{proof}
(1). For sake of contradiction, assume that there exists $k$ such that
no interval of length $k$ is included in $A$.
By the hypothesis $\overline{d}(A)=1$ there exist
arbitrarily large $N$ such that $|A\cap[1,N]|/N>1-1/2k$.
Write $N=Mk+h$ where $h<k$.
Since no interval $[ki+1,ki+k]$ is included in $A$,
the following inequalities hold:
$$\frac{|A\cap[1,N]|}{N}\le
\frac{1}{Mk}\sum_{i=0}^{M-1}|A\cap[ki+1,ki+k]|+\frac{h}{Mk}\le
\frac{1}{Mk}\sum_{i=0}^{M-1}(k-1)+\frac{1}{M}=1-\frac{1}{k}+\frac{1}{M}.$$
If we take $N$ sufficiently large then $1/M<1/2k$,
and we obtain the contradiction $|A\cap[1,N]|/N<1-1/2k$.

\smallskip
(2). Pick any $m_1\in A$;
then pick an interval $[n_2,m_2]\subseteq A$ of lenght $>m_1$;
then pick an interval $[n_3,m_3]$ of lenght $>m_1+m_2$; and
iterate the process at step $k$ by taking an interval
$[n_{k+1},m_{k+1}]\subseteq A$ of length $>m_1+\ldots+m_k$.
It is easily verified that if $X=(m_k)$ then the difference set
$\Delta(X)\subseteq A$.

\smallskip
(3). Given $X=\{x_1<x_2<\ldots\}$ consider the family $\{A-x_i\mid i\in\N\}$ 
of shifted sets $A-x_i:=\{n\in\N\mid n+x_i\in A\}$.
Notice that for every $i$, the upper density $\overline{d}(A-x_i)=\overline{d}(A)=\alpha>0$.
If the sets $A-x_i$ are pairwise disjoint for $i=1,\ldots, N$, then it is easily verified that
$$\overline{d}((A-x_1)\cup\ldots\cup(A-x_N))= 
\overline{d}(A-x_1)+\ldots+\overline{d}(A-x_N)=N\alpha.$$
If we pick $N>1/\alpha$ we obtain a contradiction.
We conclude that there exist $1\le i<j\le N$ such that $(A-x_i)\cap(A-x_j)\ne\emptyset$,
and hence $x_j-x_i\in \Delta(A)\cap\Delta(X)\ne\emptyset$.

\smallskip
(4). 
Let $X=\{x_1<x_2<\ldots\}$ be an infinite set such that $\Delta(X)\subseteq A$.
Let $[\N]^2=D_1\cup\ldots\cup D_r$ be the finite partition where
$\{n<m\}\in D_j\Leftrightarrow x_m-x_n\in C_j$. 
By \emph{Ramsey's Theorem} for pairs there exists $i$
and an infinite $H$ such that $[H]^2\subseteq D_i$. 
If we let $Y=\{x_h\mid h\in H\}$ then clearly $\Delta(Y)\subseteq C_i$.
\end{proof}

A simple use of the above properties yields the following

\begin{proposition}
\

\begin{enumerate}
\item
Let $A$ be a $\Delta$-set with positive upper density.
Then for every $\ell$ there exists a triple $\{b, c, b+\ell c\}\subseteq A$.
\item
For every finite coloring $\N=C_1\cup\ldots\cup C_r$ and for every $\ell$,
there exists a monochromatic triple $\{b, c, b+\ell c\}\subseteq C_i$.
\end{enumerate}
\end{proposition}

\begin{proof}
(1). Let $X$ be an infinite set such that $\Delta(X)\subseteq A$.
Given $\ell$, consider the set $\ell X=\{\ell x\mid x\in X\}$. 
By the previous proposition, the intersection 
$\Delta(A)\cap\Delta(\ell X)\ne\emptyset$,
and so there exist elements $a'>a$ in $A$, and $x'>x$ in $X$
such that $a'-a=\ell x'-\ell x$. Then $b:=a\in A$,
$c:=x'-x\in A$, and $b+\ell c=a'\in A$.

\smallskip
(2). Without loss of generality we can assume that 
$\overline{d}(C_j)>0$ for $j=1,\ldots,s$ and $\overline{d}(C_j)=0$ for $j=s+1,\ldots,r$.
Then $C_{s+1}\cup\ldots\cup C_r$ has upper density $0$, and hence
$B=C_1\cup\ldots\cup C_s$ has lower density $1$.
In particular, $B$ is thick, and hence a $\Delta$-set.
Since $\Delta$-sets are partition regular,
one of the sets of positive density $C_i$ is 
also a $\Delta$-set. Then apply the previous point.
\end{proof}

\begin{remark}
The idea of choosing 
a color that is simultaneously of positive asymptotic density 
and a $\Delta$-set is certainly not new.
As far as the authors know, it was first used by 
V. Bergelson in \cite{be}.
\end{remark}

As a consequence of the previous proposition we obtain
the desired ultrafilters.

\begin{theorem}
The following family has the finite intersection property:\footnote
{~Sets $A$ with the property that $A\cap\Delta(X)\ne\emptyset$ for every
infinite $X$ are named $\Delta^*$-\emph{sets}.}
$$\F=\{A\subseteq\N\mid \underline{d}(A)=1\}\cup
\{A\subseteq\N\mid A\cap\Delta(X)\ne\emptyset \ \text{for every infinite}\ X\},$$
and every ultrafilter $\V$ that extends $\F$ satisfies property $(\star)$
of Theorem \ref{theorem1}.
\end{theorem}

\begin{proof}
Trivially $\N\in\F$, and $\F$ is closed under supersets.
Now assume, for sake of contradiction, that there are $A_1,\ldots,A_k\in\F$ 
with $A_1\cap\ldots\cap A_k=\emptyset$. By taking complements,
we obtain the finite coloring $\N=(A_1)^c\cup\ldots\cup(A_k)^c$.
By the above proposition, there exists a color $(A_i)^c$
that has positive upper density and includes a set 
of differences $\Delta(Y)$ for an infinite $Y$.
But then $\underline{d}(A_i)<1$ and $A_i\cap\Delta(Y)=\emptyset$,
contradicting $A_i\in\F$.

Now let $\V$ be any ultrafilter that extends $\F$ and let $A\in\V$.
Then $A$ is a $\Delta$-set with positive upper density.
Indeed, $A$ must be a $\Delta$-set, as otherwise its complement
$A^c\in\F$, and we would have $\emptyset=A\cap A^c\in\V$.
Similarly, $A$ must have positive upper density, as
otherwise $\underline{d}(A^c)=1$ and hence $A^c\in\F$,
and we would have $\emptyset=A\cap A^c\in\V$.
\end{proof}

\medskip
\section{Infinite exponential patterns}

In this second section we show how the arguments used in the previous section 
can be extended to produce \emph{infinite} monochromatic exponential configurations.
It is worth noting that the possibility of constructing \emph{infinite} patterns 
is one of the distinctive strengths of the ultrafilter technique in Ramsey Theory.
We observe that all monochromatic configurations proved in \cite{sa},
while possessing a richer structure that also involves multiplications, are all \emph{finite}.

\begin{theorem}\label{theorem2}
Let $f:\N\times\N\to\N$ be the function $f(n,m)=2^n m$
and let $\U$ be any ultrafilter on $\N$ that satisfies the following property:
\begin{itemize}
\item[($\dagger$)]
For every $A\in\U$ and for every $L$ there exist elements $b,c$ such that
$b, c, b+\ell c\in A$ for every $\ell\le L$.
\end{itemize}
Then for every $A\in f_*(\U\otimes\U)$
there exists an infinite sequence $(a_n)$
with the property that for all $i,j,k$ where $i<2j$ and $2j+1<k$, the triple
$\{\,x:=2^{a_i}a_{2j};\ y:=2^{a_{2j+1}}a_k;\ 2^x y\,\}\subseteq A$.
\end{theorem}

Note that $(\star)$ is the weaker version of the $(\dagger)$ property 
in which, instead of considering every $\ell\le L$, only a single 
$\ell$ is considered.

We remark that the existence of ultrafilters with the property $(\dagger)$ is not a problem, 
since every idempotent minimal ultrafilter of $(\beta\N,\oplus)$ satisfies it.
Since we have stated that we do not use the properties of the algebra of ultrafilters, 
we give below a direct proof that is obtained from Brauer's Theorem, 
an improvement of van der Waerden's Theorem in which even the common 
difference of the monochromatic arithmetic progression has the same color.

\begin{itemize}
\item
\emph{Brauer's Theorem} (1929).
For every finite coloring $\N=C_1\cup\ldots\cup C_r$ there
exists a color $C_i$ such that for every $\ell$ one finds
a monochromatic pattern $\{b, c, b+c, \ldots, b+\ell c\}\subseteq C_i$.
\end{itemize}

We are not aware of any proof of Brouer's Theorem that 
only uses simple arguments grounded on Ramsey's Theorem and asymptotic density.

\begin{theorem}
There exist ultrafilters $\U$ that satisfy property $(\dagger)$.
\end{theorem}

\begin{proof}
For every $A\subseteq\N$ and for every $\ell$, let
$$\Gamma(A,\ell):=\{b\in\N\mid \exists c\in\N\ \text{s.t. either}\ 
b, c, b+\ell c\in A\ \text{or}\ b, c, b+\ell c\notin A\}.$$
The family $\mathcal{G}:=\{\Gamma(A,\ell)\mid A\subseteq\N,\ \ell\in\N\}$ has the FIP.
Indeed, given sets $\Gamma(A_j,\ell_j)$ for $j=1,\ldots,k$,
let $\ell=\max\{\ell_j\mid j=1,\ldots,k\}$, and let
$\N=C_1\cup\ldots\cup C_r$ be the finite coloring induced
by the family $\{A_1,\ldots,A_k\}$.\footnote
{~Precisely, the colors $C_i$ are obtained as the intersections
$\left(\bigcap_{j\in F}A_j\right)\cap\left(\bigcap_{j\notin F}(A_j)^c\right)$
for all $F\subseteq\{1,\ldots,k\}$. 
(We agree that $\bigcap_{j\in\emptyset}A_j=\bigcap_{j\in\emptyset}(A_j)^c=\N$.)}
By Brauer's Theorem there exists a color $C_i$ and a
monochromatic pattern $\{b, c, b+c, \ldots, b+\ell c\}\subseteq C_i$.
Then it is readily seen that the pair $(b,c)\in\bigcap_{j=1}^k\Gamma(A_j,\ell_j)$.
Finally, it is verified in a straightforward manner that
every ultrafilter $\U$ that extends the family $\mathcal{G}$
has the desired property $(\dagger)$.
\end{proof}

Before proving the above Theorem \ref{theorem2}, let us see two relevant
consequences of its about the existence of infinite monochromatic configurations.

\begin{corollary}\label{corollary2a}
For every finite coloring $\N=C_1\cup\ldots\cup C_r$ 
there exists a color $C_s$ and an infinite sequence $(a_n)$
such that for all $i,j,k$ with $i<2j$ and $2j+1<k$, the triple 
$\{\,x:=2^{a_i}a_{2j};\ y:=2^{a_{2j+1}}a_k;\ 2^x y\,\}\subseteq C_s$.

In particular, by letting $b_n:=2^{a_{2n-1}}a_{2n}$, 
we get the existence of an
infinite sequence $(b_n)$ such that
all elements $b_n, 2^{b_n}b_{n+1}\in C_s$ are monochromatic.
\end{corollary}

\begin{proof}
Pick an ultrafilter $\U$ with property $(\dagger)$
and let $C_s$ be the color that belongs to 
the ultrafilter $f_*(\U\otimes\U)$.
\end{proof}

\begin{corollary}\label{corollary2b}
For every finite coloring $\N=C_1\cup\ldots\cup C_r$ 
there exists a color $C_s$ and an infinite sequence $(a_n)$
such that for all $i,j,k$ with $i<2j$ and $2j+1<k$, the triple
$\{\,a:=2^{2^{a_i}a_{2j}};\ b:=2^{2^{a_{2j+1}}a_k};\ b^a\,\}\subseteq C_s$.

In particular, by letting $b_n:=2^{2^{a_{2n-1}}a_{2n}}$, 
we get the existence of an
infinite sequence $(b_n)$ such that
all elements $b_n, (b_{n+1})^{b_n}\in C_s$ are monochromatic.
\end{corollary}

\begin{proof}
It directly follows by applying the previous corollary
to the finite coloring $\N=D_1\cup\ldots\cup D_r$ 
where $D_s=\{n\mid 2^n\in C_s\}$. Indeed, pick 
an infinite sequence $(a_n)$ such that for all $i,j,k$ with $i<2j$ and $2j+1<k$, 
the triple $\{\,x:=2^{a_i}a_{2j};\ y:=2^{a_{2j+1}}a_k;\ 2^x y\,\}\subseteq D_s$.
Then it is readily verified that $(a_n)$ satisfies the desired properties.
\end{proof}

\begin{proof}[Proof of Theorem \ref{theorem2}]
By the definitions, $A\in f_*(\U\otimes\U)$ if and only if
$\widehat{A}:=\{n\in\N\mid A/{2^n}\in\U\}\in\U$.
Pick $a_1\in A'$, so that $A/{2^{a_1}}\in\U$.
Then $A_1:=A'\cap A/{2^{a_1}}\in\U$ and we can pick 
$a_2, a_3$ such that 
$a_3, a_2, a_3+\ell a_2\in A_1$ for every $\ell\le2^{a_1}$. 
In consequence, $2^{a_1}a_2, 2^{a_1}a_3\in A$
and $A/{2^{a_2}}, A/{2^{a_3}}, A/{2^{a_3+2^{a_1}a_2}}\in\U$.
Then $A_2:=A'\cap A/{2^{a_1}}\cap A/{2^{a_2}}\cap A/{2^{a_3}}\cap A/{2^{a_3+2^{a_1}a_2}}\in\U$,
and we can pick $a_4, a_5$ such that 
$a_5, a_4, a_5+\ell a_4\in A_2$ for every $\ell\le\max\{2^{a_1},2^{a_2},2^{a_3}\}$.
In consequence, $2^{a_i}a_4, 2^{a_i}a_5\in A$ for $i=1,2,3$, and
$2^{a_3+2^{a_1}a_2}a_4, 2^{a_3+2^{a_1}a_2}a_5\in A$, and 
$$A/{2^{a_4}}, A/{2^{a_5}}, A/{2^{a_5+2^{a_1}a_4}},A/{2^{a_5+2^{a_2}a_4}},
A/{2^{a_5+2^{a_3}a_4}}\in\U.$$ 

Now inductively proceed in this way, and assume that elements
$a_1,\ldots,a_{2n-1}$ have been defined in such a way that:
\begin{enumerate}
\item
$a_i\in A'$ for every $i\le 2n-1$;
\item
$a_{2j+1}+2^{a_i}a_{2j}\in A'$ for all $i<2j<2n-1$.
\item
$2^{a_i}a_k\in A$ for all $i<k\le 2n-1$, except when
$k=i+1$ is odd.
\item
$2^{a_{2j+1}+2^{a_i}a_{2j}}a_k\in A$ for all $i,j,k$ where
$i<2j$ and $2j+1<k\le 2n-1$.
\end{enumerate}

At the inductive step, consider
$$A_n:=A'\cap\bigcap_{i=1}^{2n-1}A/{2^{a_i}}\cap
\bigcap_{i<2j<2n-1} A/{2^{a_{2j+1}+2^{a_i}a_{2j}}}.$$
By properties (1) and (2), the set $A_n\in\U$.
Then we can pick $a_{2n}, a_{2n+1}$
such that $a_{2n+1}, a_{2n}, a_{2n+1}+\ell a_{2n}\in A_n$ 
for every $\ell\le\max\{2^{a_i}\mid i\le 2n-1\}$. Then all required
properties are satisfied by the sequence
$a_1,\ldots,a_{2n-1},a_{2n},a_{2n+1}$
obtained by adding the new two elements $a_{2n}, a_{2n+1}$. Indeed,

\begin{enumerate}
\item
$a_{2n},a_{2n+1}\in A'$, and hence
$a_i\in A'$ for every $i\le 2n+1$.
\item
$a_{2n+1}+2^{a_i}a_{2n}\in A'$ for every $i\le 2n-1$,
and hence
$a_{2j+1}+2^{a_i}a_{2j}\in A'$ for all $i<2j<2n+1$.
\item
$a_{2n},a_{2n+1}\in A/{2^{a_i}}\Leftrightarrow 2^{a_i}a_{2n}, 2^{a_i}a_{2n+1}\in A$ 
for every $i\le 2n-1$ (but in general $2^{a_{2n}}a_{2n+1}\notin A$),
and hence $2^{a_i}a_k\in A$ for all $i<k\le 2n+1$,
except when $k=i+1$ is odd.
\item
$a_{2n},a_{2n+1}\in A/{2^{a_{2j+1}+2^{a_i}a_{2j}}}\Leftrightarrow
2^{a_{2j+1}+2^{a_i}a_{2j}}a_{2n},2^{a_{2j+1}+2^{a_i}a_{2j}}a_{2n+1}\in A$, and hence
$2^{a_{2j+1}+2^{a_i}a_{2j}}a_k\in A$ for all $i,j,k$ where
$i<2j$ and $2j+1<k\le 2n+1$.
\end{enumerate}

By the above construction, we obtain an infinite sequence $(a_n)$
with the desired property. Indeed, let $i,j,k$ be such that
$i<2j$ and $2j+1<k$. Then $x:=2^{a_i}a_{2j}\in A$ and
$y:=2^{a_{2j+1}}a_k\in A$ by (3); moreover, also
$2^xy=2^{a_{2j+1}+2^{a_i}a_{2j}}a_k\in A$ by (4).
\end{proof}

\bibliographystyle{amsalpha}

\end{document}